\documentclass[11pt]{article}

\usepackage{mathrsfs}
\usepackage{mathrsfs,amsfonts}
\usepackage[dvips]{color}
\usepackage{amsmath}
\usepackage{comment}
\usepackage[final]{pdfpages}
\usepackage{graphicx}
\usepackage{float}
\usepackage{epstopdf}
\usepackage[francais,english]{babel}
\usepackage{subfigure}
\usepackage{epsf}

 \newtheorem{theorem}{Theorem}[section]
 \newtheorem{lemma}[theorem]{Lemma}
 \newtheorem{corol}[theorem]{Corollary}
 \newtheorem{prop}[theorem]{Proposition}
 \newtheorem{example}[theorem]{Example}
 \newtheorem{condition}[theorem]{Condition}

 \def\bitemize{\begin{itemize}}\def\eitemize{\end{itemize}}

 \def\beqlb{\begin{eqnarray}}\def\eeqlb{\end{eqnarray}}
 \def\beqnn{\begin{eqnarray*}}\def\eeqnn{\end{eqnarray*}}

 \def\qed{\hfill$\Box$\medskip}

\begin{document}
\selectlanguage{english}

\

\bigskip\bigskip\bigskip

\centerline{\LARGE\bf Toroidal dimer model and Temperley's bijection}

\bigskip

\centerline{Wangru Sun}

\bigskip

\medskip

{\narrower

\noindent{\bf Abstract.} Temperley's bijection relates the toroidal dimer model to cycle rooted spanning forests ($CRSF$) on the torus. The height function of the dimer model and the homology class of $CRSF$ are naturally related. When the size of the torus tends to infinity, we show that the measure on $CRSF$ arising from the dimer model converges to a measure on (disconnected) spanning forests or spanning trees. There is a phase transition, which is determined by the average height change.

\smallskip

\smallskip




\bigskip

\section{Introduction}

The dimer model, also called the perfect matching model, was first introduced in physics and chemistry to model the adsorption of di-atomic molecules on the surface of a crystal \cite{FR37}. In the 1960's Kasteleyn (\cite{Kas61}, \cite{Kas67}), Fisher and Temperley (\cite{TF61}) have shown how to calculate the partition function. Many progresses have been made since the late 1990s, for example, \cite{Ken97}, \cite{CKP01}, \cite{KO2}, etc. In \cite{KOS06}, the authors reveal the existence of a phase diagram for the dimer model on infinite bipartite bi-periodic graphs.

A spanning tree of a graph is a connected, contractible union of edges where every vertex is covered. Pemantle (\cite{Pem}) proves that the uniform spanning tree measures on finite subgraphs of $\mathbb{Z}^d$ converge weakly, as the subgraphs tend to the whole of $\mathbb{Z}^d$. When $d\leq4$, the limiting measure is supported on a spanning tree of $\mathbb{Z}^d$, otherwise there are almost surely infinitely many trees. Burton and Pemantle (\cite{BP}) prove a transfer current theorem.

Temperley (\cite{Tem74}) first introduced a bijection on the square grid between spanning trees and dimer configurations. It was generalized by Burton and Pemantle in \cite{BP} to unweighted planar graphs. Kenyon, Propp and Wilson (\cite{KPW}) generalized this construction to directed weighted planar graphs by providing a measure preserving bijection between oriented weighted spanning trees of the planar graph and dimer configurations of its double graph, see Section 2 for definitions.

Let $G$ be a bi-periodic planar graph, and for $N\in\mathbb{N}^*$, consider the toroidal graph $\mathcal{G}_N=G/(N\mathbb{Z})^2$. Then, Temperley's bijection relates dimer configurations of its double graph to \textit{Cycle Rooted Spanning Forests} ($CRSF$) of $\mathcal{G}_N$. On $CRSF$ there is a natural probability measure arising from that of the dimer model.

A dimer configuration gives a height function. If the graph is toroidal, the height function is additively multivalued. Dub\'edat and Gheissari (\cite{Dub}) show that, under Temperley's bijection, the height function of the dimer model and the homology class of the $CRSF$ are naturally related. Proposition 2.1 of this paper gives an independent proof, which relies on geometric considerations.

When $N\rightarrow\infty$, such measures on $CRSF$ converge to a limiting Gibbs measure~$\mu$.

A natural question is the topology of the support of the limiting measure. In this paper we give a characterization of the number of connected components. It can loosely be stated as follows. A precise statement is given in Theorem 4.2 and Theorem 5.3.

\begin{theorem}
When the slope of the limiting dimer measure is non-zero, then under $\mu$, there are a.s. infinitely many connected components.
\end{theorem}

\begin{theorem}
Let $G$ be a graph verifying the condition ($\star$). When the slope of the limiting dimer measure is zero, then under $\mu$, there is a.s. one connected spanning tree.
\end{theorem}

For the definition of the condition ($\star$), see Section 5. Especially, this condition is verified by the drifted square grid graph, see Example 5.6 for definition. This name is inherited form \cite{Chhita}.

Combining Theorems 1.1 and 1.2 gives a full picture of the phase diagram. When the slope of the limiting dimer measure is not zero, there are a.s. infinitely many trees, and when the slope is zero, there is a.s. only one spanning tree. Zero magnetic field lies in the connected phase. In the case of the drifted square grid graph, this can be pictured as in Figure 1. A more detailed statement is given in Section 6.

\begin{figure}[H]
\centering
\includegraphics[width=0.5\textwidth]{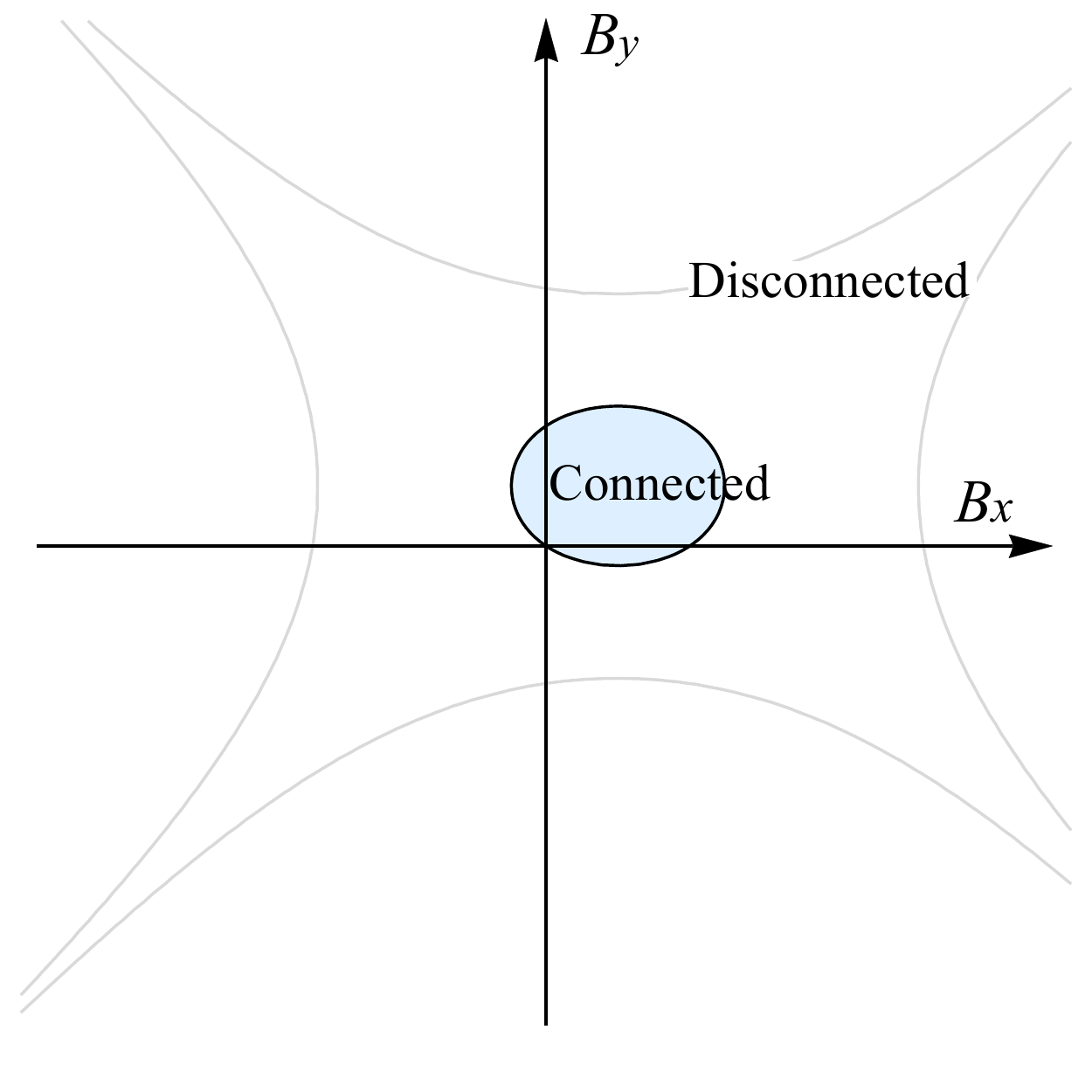}
\caption{Phase diagram of a typical weighting.}
\end{figure}

\textbf{Acknowledgements.} I would like to thank C\'edric Boutillier and B\'eatrice de Tili\`ere for their directions, comments and references. Also it's my pleasure to thank Richard Kenyon for his insightful comments.

\section{Definitions and Facts}
\subsection{Basic structures}

Let $G=(V,E)$ be a planar connected graph, where $V$ is the set of vertices and $E$ is the set of edges. We take $G^*=(V^*, E^*)$ as its \textit{dual}, whose vertices correspond to faces of $G$ and two vertices are joined by an edge in $E^*$ if and only if these two faces are neighboring in $G$ ($G$ is also called the \textit{primal}). If we take the union of $G$ and $G^*$, color $V$ and $V^*$ in black (in the figures we use grey diamonds to represent vertices in $V^*$), take the intersections of the edges as vertices and color them in white, the new graph we obtain is denoted by $G^d$ and called the \textit{double graph} of $G$. Every black vertex of $G^d$ has only white neighbors and vice-versa. Such property is called \textit{bipartite}.

\begin{figure}[H]
\centering
\includegraphics[width=0.6\textwidth]{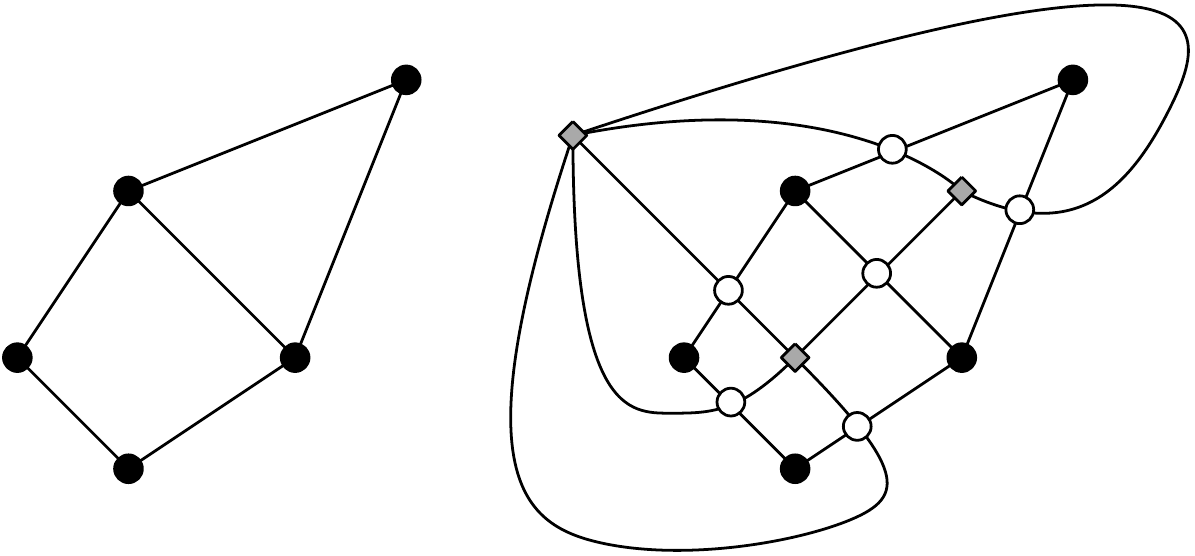}
\caption{Graph $G$ and its double $G^d$.}
\end{figure}

If $G$ is an infinite $\mathbb{Z}^2$-periodic graph, then its quotient graph of size $N$ is the toroidal graph $\mathcal{G}_N=G/(N\mathbb{Z})^2$. We use $\mathcal{G}^d_N$ and $\mathcal{G}^*_N$ to denote the quotient graphs of $G^d$ and $G^*$. Note that, in the notation for graphs, calligraphic letters (like $\mathcal{G}$) symbolize toroidal graphs, and normal letters (like $G$) symbolize planar ones or both of them (when we talk about something for both planar and toroidal graphs).

We say that a graph $G$ is \textit{weighted} and \textit{directed} if every directed edge $uv$ of $G$ is assigned a non-negative weight, which in general is different from that of $vu$. A weight function $c$ is a non-negative function defined on directed edges of $G$. By saying that a graph is \textit{unweighted}, we mean that all edges have weight $1$. For $G^*$ arising from $G$, by default we set $G^*$ to be unweighted.

We say that $G^d$ is \textit{weighted} if every (non-directed) edge of $G^d$ is assigned a non-negative weight. We denote this weight function by $c$ again, and in general there is no ambiguity when we use the same letter $c$ to denote weight functions defined on different objects.

There exits a bijection between the weight functions on $G$ and $G^*$ (as weighted and directed graphs) and the weight functions on $G^d$ (as a weighted graph). For every edge $uv$ of $G$, on $G^d$ let $w$ be the white vertex between $u$ and $v$ (as in figure 3). Given a weight function $c$ on $G$ and $G^*$, we let $c(uw)$ be $c(uv)$ and let $c(vw)$ be $c(vu)$. In the same way we assign a weight for every edge arising from $G^*$. This bijection is to be used in the setting of Temperley's bijection, see Section 2.2.

\begin{figure}[H]
\centering
\subfigure[Weights of $G$.]{
\includegraphics[width=0.25\textwidth]{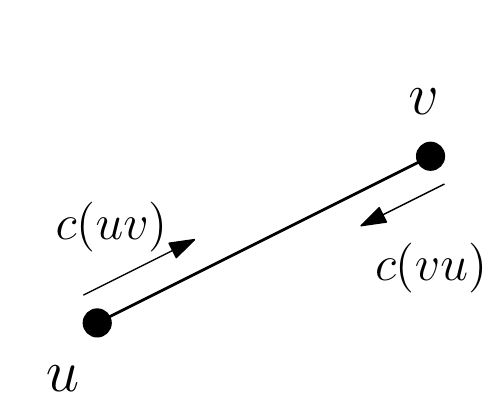}}
\subfigure[Weights of $G^d$.]{
\includegraphics[width=0.25\textwidth]{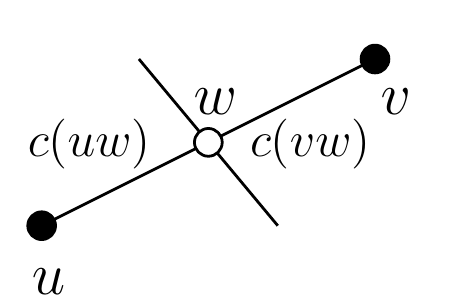}}
\caption{Weights.}
\end{figure}

An \textit{oriented spanning tree} ($OST$) of a graph $G$ is a connected, contractible union of directed edges such that every vertex of $G$ except one has exactly one outgoing edge. The only vertex having no outgoing edge is called the root of the tree. The weight of the tree is the product of the weights of the edges.

An \textit{oriented cycle-rooted spanning forest} ($OCRSF$) of a toroidal graph $\mathcal{G}$ is a union of directed edges such that every vertex of $\mathcal{G}$ has exactly one outgoing edge, and edges don't form contractible cycles. Each connected component of an $OCRSF$ is called \textit{oriented cycle-rooted tree} ($OCRT$), which contains exactly one non-trivial oriented cycle, and every edge other than those on the cycle is oriented towards the cycle. This cycle is called the root-cycle of the $OCRT$. For each configuration, the root-cycles are all parallel (in the sense of homotopy).

A \textit{dimer configuration} of a bipartite graph is a subset of edges such that every vertex is incident to exactly one edge in the subset. The weight of a dimer configuration is the product of the weights of edges present. We denote the set of all dimer configurations of a graph by $\mathcal{M}$.

For spanning trees and dimer configurations, we can always define a probability measure arising from the weighting, where the probability of a configuration is proportional to its weight.

The key object for calculating the partition function of the dimer model is the \textit{Kasteleyn matrix}, see \cite{Kas61} for example. A \textit{Kasteleyn orientation} of a graph is an orientation of edges such that when traveling clockwise around the boundary of a face, the number of co-oriented edges is odd. If the graph is weighted and bi-partite, the Kasteleyn matrix $K$ has rows indexed by black vertices, columns indexed by white vertices, and coefficients defined by:
\begin{eqnarray*}
K_{b,w}=
\begin{cases}
c(bw) \quad &\text{if }b\sim w,\ b\rightarrow w\\
-c(bw) \quad &\text{if }b\sim w,\ b\leftarrow w\\
0 \quad &\text{else,}
\end{cases}
\end{eqnarray*}
where $b$ is a black vertex and $w$ is a white one.

The dimer partition function of a planar graph is equal to $\det K$ up to a sign. For a toroidal graph, choose a simple curve $\gamma_x$ (resp. $\gamma_y$) on the dual of the graph which winds once horizontally (resp. vertically) around the torus. For every edge crossing $\gamma_x$, multiply the corresponding entry by $z$ if its black end is on the left of $\gamma_x$ and by $z^{-1}$ if the white end is on the left, respectively $w$ or $w^{-1}$ for edges crossing $\gamma_y$. Such modified Kasteleyn matrix is denoted by $K(z,w)$. The \textit{characteristic polynomial} is defined as
\begin{eqnarray}
P(z,w)=\det K(z,w).
\end{eqnarray}
\begin{figure}[H]
\centering
\includegraphics[width=0.7\textwidth]{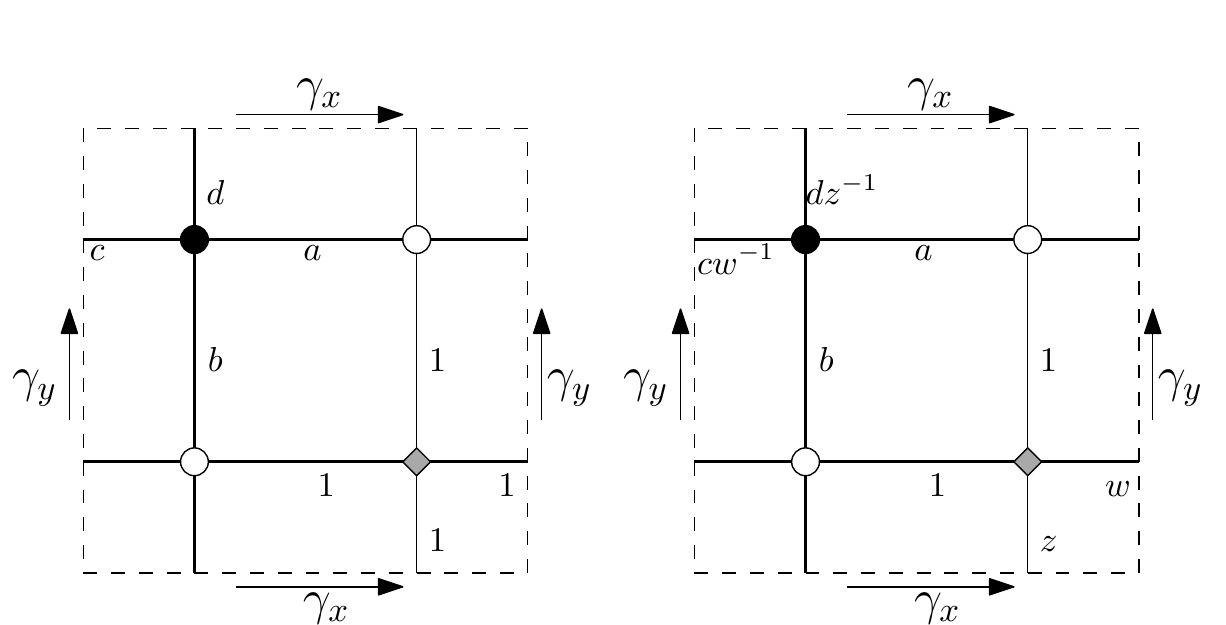}
\caption{Adding $z$ and $w$ for a toroidal double graph.}
\end{figure}

If we let $z=(-1)^{\theta}$ and $w=(-1)^\tau$ and we denote the corresponding Kasteleyn matrix by $K^{(\theta,\tau)}$, then the partition function is a linear combination of $K^{(\theta,\tau)}$, see \cite{Kas67} or \cite{CKP01} for example. After choosing a proper Kasteleyn orientation, we have
\begin{eqnarray}
Z=\frac{1}{2}(-\det K^{(0,0)}+\det K^{(0,1)}+\det K^{(1,0)}+\det K^{(1,1)}).
\end{eqnarray}
\subsection{Temperley's bijection}

The authors of \cite{KPW} construct a general version of Temperley's bijection for directed weighted planar graphs. The construction also applies to graphs on other surfaces. In this paper, we use Temperley's bijection on toroidal graphs and on planar graphs.

\subsubsection{Planar case}

We begin by the planar case. Let $G$ be a planar graph. Suppose that a vertex $v_0$ is incident to a face $f_0$. If $v_0$, $f_0$ and the edges of $G^d$ incident to them are taken away, then we denote the rest of $G^d$ by $G^d(v_0, f_0)$.

\textit{Temperley's bijection} (\cite{KPW}) is defined as a mapping between dimer configurations of $G^d(v_0,f_0)$ and spanning trees of $G$ rooted at $v_0$ (in fact spanning-tree pairs of $G$ and $G^*$), given by the following procedure. We start from a spanning tree $T$ of $G$ rooted at $v_0$. For $G^*$, edges not crossing $T$ form a connected configuration without cycle, thus a tree. We denote it by $T^*$ and call it the dual of $T$. Take $f_0$ as its root. The set of such pairs $(T,T^*)$ is denoted by $\mathcal{T}(G,G^*)$.

\begin{figure}[H]
\centering
\includegraphics[width=0.7\textwidth]{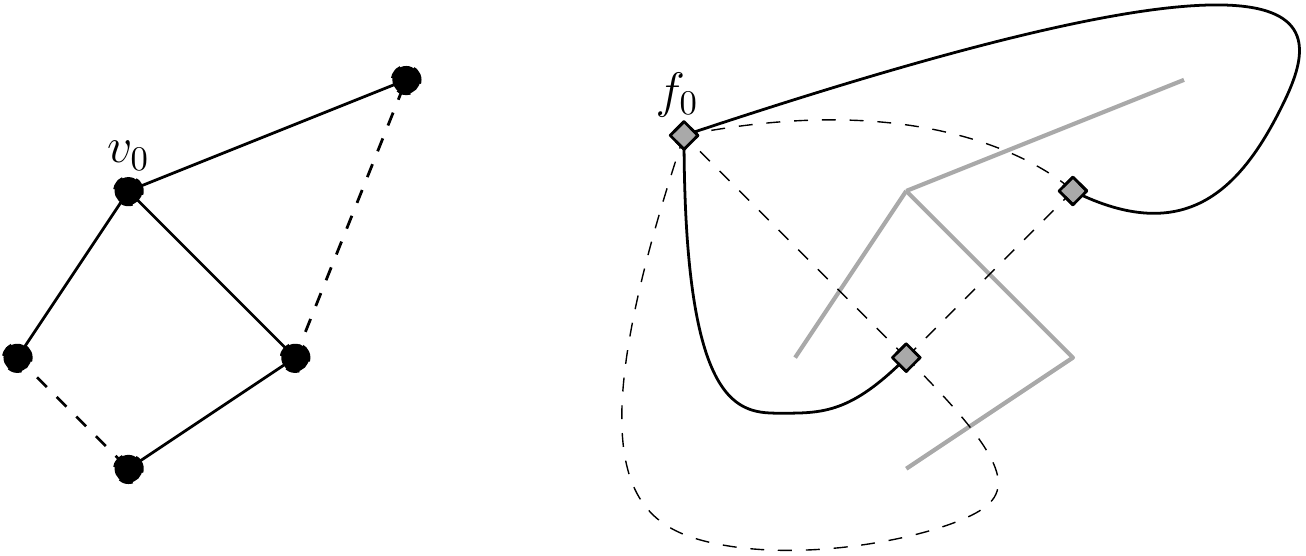}
\caption{Rooted spanning tree $T$ and its dual $T^*$.}
\end{figure}

Let $M$ be a subset of edges of $G^d$. An edge $uw$ of $G^d$ is in $M$ if the directed edge $uv$ is in $T$ or $T^*$, where $w$ is the white vertex of $G^d$ between $u$ and $v$. Edges in $M$ form a perfect matching of $G^d(v_0, f_0)$.

\begin{figure}[H]
\centering
\subfigure[Part of a spanning tree $T$.]{
\includegraphics[width=0.3\textwidth]{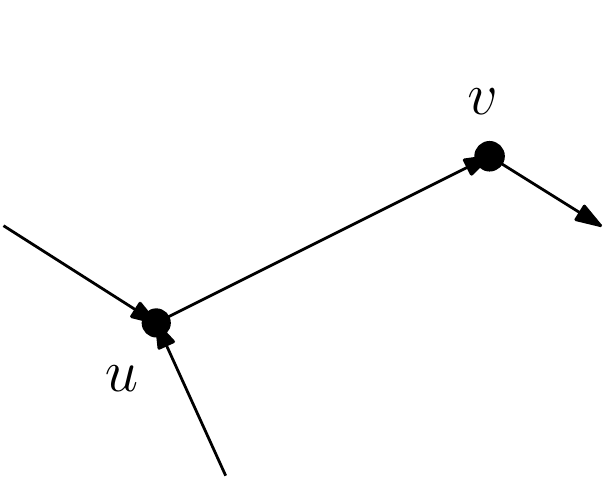}}
\subfigure[Corresponding dimer configuration.]{
\includegraphics[width=0.3\textwidth]{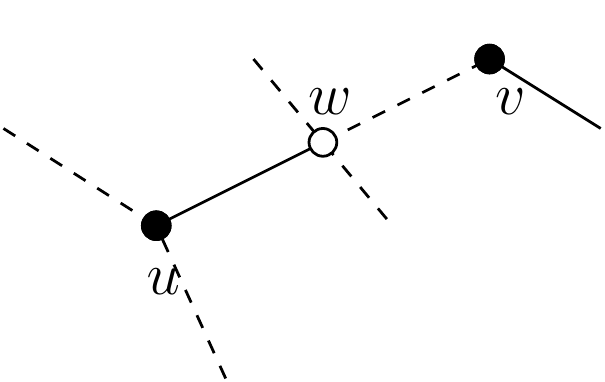}}
\caption{Temperley's bijection from a spanning tree to a perfect matching.}
\end{figure}

\begin{figure}[H]
\centering
\includegraphics[width=0.7\textwidth]{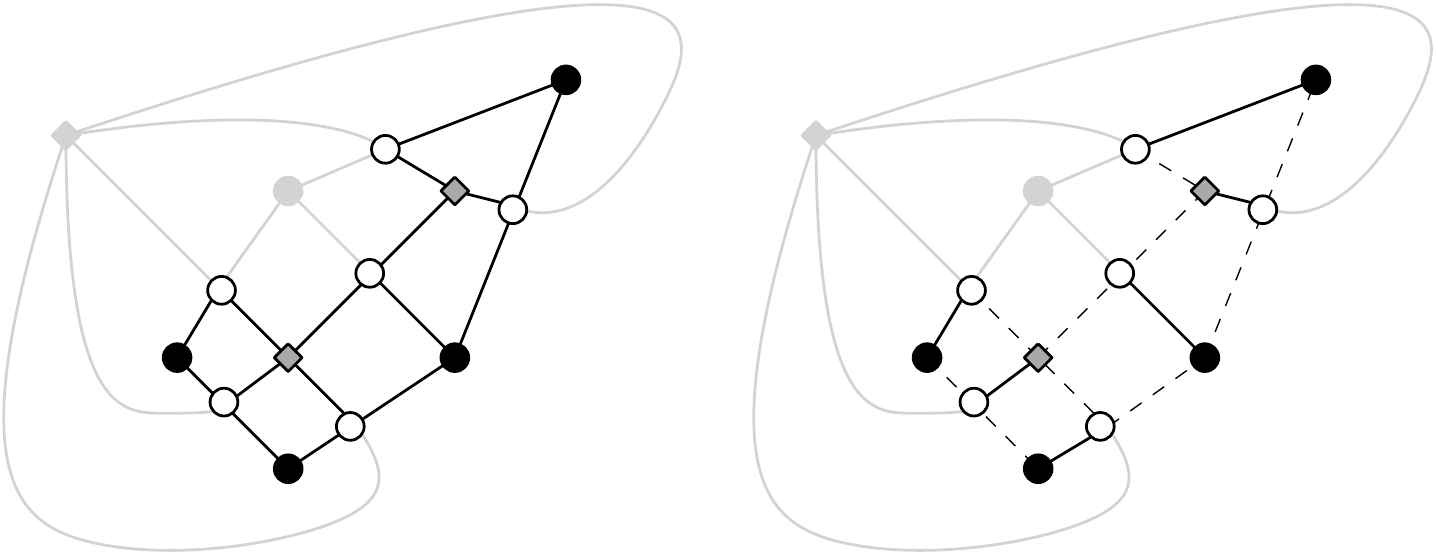}
\caption{Graph $G^d(v_0,f_0)$ and the perfect matching arising from $T$ and $T^*$ in Figure 5.}
\end{figure}

By the same rule, from any perfect matching $M\in\mathcal{M}\big(G^d(v_0, f_0)\big)$, we can build a spanning-tree-pair $(T,T^*)\in\mathcal{T}(G,G^*)$. A directed edge $uv$ of $G$ (resp. $G^*$) is in $T$ (resp. $T^*$) if $uw$ is in $M$.

In \cite{KPW}, the authors prove that such a map is a measure preserving bijection: if Temperley's bijection relates $M\in\mathcal{M}\big(G^d(v_0, f_0)\big)$ to $(T,T^*)\in\mathcal{T}(G,G^*)$, then the probability of $M$ is equal to the probability of $(T,T^*)$, which is defined as
\begin{eqnarray}
\frac{\prod_{\vec{e}\in T}c({\vec{e}})\prod_{{\vec{e^*}}\in T^*}c({\vec{e^*}})}{Z_{\mathcal{T}(G,G^*)}}.
\end{eqnarray}
Here the partition functions $Z_{\mathcal{T}(G,G^*)}$ is the sum of $\prod_{\vec{e}\in T}c({\vec{e}})\prod_{{\vec{e^*}}\in T^*}c({\vec{e^*}})$ over all pairs $(T,T^*)$.

\subsubsection{Toroidal case}

For any toroidal graph $\mathcal{G}$, let $F$ be an $OCRSF$ of $\mathcal{G}$ and $F^*$ be an $OCRSF$ of $\mathcal{G}^*$. We say that $F^*$ is a dual of $F$ if $F$ and $F^*$ don't cross. If $F$ has $k$ connected components, then it has $2^k$ duals. Every dual of $F$ has also $k$ components, and its root-cycles are parallel to those of $F$. We denote the pairs of such dual forests $(F,F^*)$ by $\mathcal{F}(\mathcal{G},\mathcal{G}^*)$. The weight of $(F,F^*)$ is defined as the product of the weights of all directed edges present. This gives rise to a probability measure on $\mathcal{F}(\mathcal{G},\mathcal{G}^*)$:
\begin{eqnarray}
\mathbb{P}(F,F^*)=\frac{\prod_{\vec{e}\in F}c(\vec{e})\prod_{\vec{e^*}\in F^*}c(\vec{e^*})}{Z_{\mathcal{F}(\mathcal{G},\mathcal{G}^*)}},
\end{eqnarray}
where the partition function is:
\begin{eqnarray*}
Z_{\mathcal{F}(\mathcal{G},\mathcal{G}^*)}=\sum_{(F,F^*)\in\mathcal{F}(\mathcal{G},\mathcal{G}^*)}\prod_{\vec{e}\in F}c(\vec{e})\prod_{\vec{e^*}\in F^*}c(\vec{e^*}).
\end{eqnarray*}

If we suppose that the weights of the edges of $\mathcal{G}^*$ are all $1$ (the by-default-setting for $\mathcal{G}^d$ arising from $\mathcal{G}$), then the probability of $(F,F^*)$ is:
\begin{eqnarray}
\mathbb{P}(F,F^*)=\frac{\prod_{\vec{e}\in F}c({\vec{e}})}{Z_{\mathcal{F}(\mathcal{G},\mathcal{G}^*)}}.
\end{eqnarray}

Summing over all $\mathcal{G}^*$, this gives rise to a probability measure on $OCRSF$ of $\mathcal{G}$ not proportional to weights. The weight of a configuration is multiplied by a factor $2^{k}$ where $k$ is the number of its connected components. Such measure, if compared to the normal weighted measure on $OCRSF$ of $\mathcal{G}$, encourages configurations to have more cycles.

Temperley's bijection between $OCRSF$ pairs of $\mathcal{G},\mathcal{G}^*$ and dimer configurations of $\mathcal{G}^d$ is defined as in the planar case. It is easy to verify that it is indeed a bijection and is measure preserving.

\subsection{Height function}

Following \cite{KPW}, given a dimer configuration of the planar bipartite graph $G^d$, we define a \textit{height function} on faces of $G^d$ as follows.

We suppose that $G^d$ is embedded. Note that every face of $G^d$ is a quadrilateral with two black vertices and two white vertices. When we say a \textit{diagonal} of a face, we mean the one linking two opposite black vertices. A dimer can be viewed as a cut on the plane. Given a dimer configuration $M$, we choose a face as base (the diagonal of this face has $0$ height), then prolong this to its neighboring diagonals by the turning angle without passing cuts. This can be prolonged to the whole plane(\cite{KPW}). The height $h^M(.)$ of a face is defined as the height of its diagonal. Note that the height function such defined depends on the embedding of the graph on the plane.

\begin{figure}[H]
\centering
\includegraphics[width=0.55\textwidth]{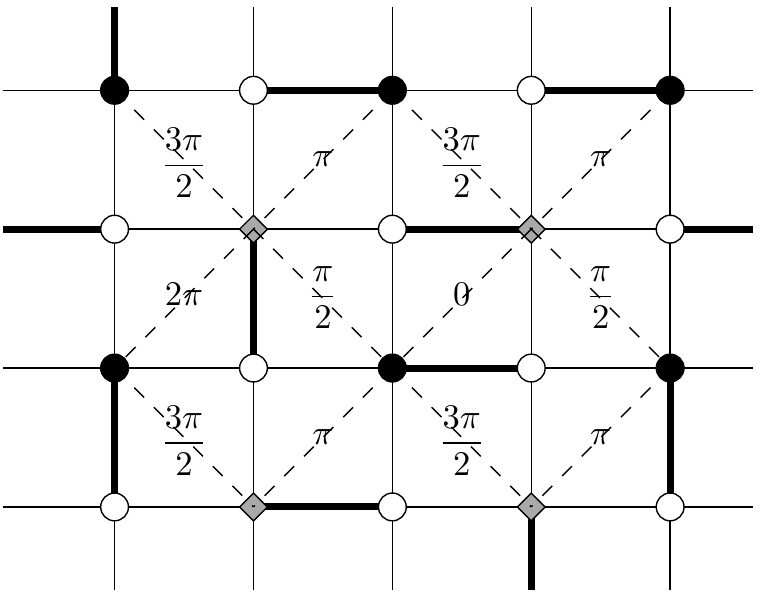}
\caption{Height function corresponding to a local dimer configuration.}
\end{figure}

Note that there is another natural definition of height function, which depends on the choice of a reference configuration $M_0$, see \cite{KOS06} for example. We denote the height function of $M$ by $\tilde{h}^{(M,M_0)}$. These two definitions are coherent:
$$\tilde{h}^{(M,M_0)}=(h^M-h^{M_0})/2\pi.$$

In the remainder of this paper, when we speak of height, we use the first definition $h$ by default.

For the infinite $\mathbb{Z}^2$-periodic planar graph $G$, its double graph $G^d$ is also $\mathbb{Z}^2$-periodic. A dimer configuration $M$ of $G^d$ gives a height function $h^M$ on the whole plane. If $M$ is also $\mathbb{Z}^2$-periodic, then it gives rise to a dimer configuration of $\mathcal{G}^d_1$. Let $(\hat{x},\hat{y})$ be a base of $\mathbb{Z}^2$. The height function $h^M$ induces a height change $(h_x^M,h_y^M)$, where $h_x^M$ (resp. $h_y^M$), for any face $f$ of $G^d$, is equal to $h^M(f+\hat{x})-h^M(f)$ (resp. $h^M(f+\hat{y})-h^M(f)$), whose value doesn't depend on the choice of $f$.

Proposition 3.1 in \cite{KOS06} shows that the characteristic polynomial (2.1) can be interpreted by the height changes as follows:
\begin{eqnarray}
P(z,w)=\sum_{M\in\mathcal{M}(\mathcal{G}^d_1)}\big(\prod_{e\in M}c(e)\big) z^{-h_x^M}w^{-h_y^M}(-1)^{h^M_x h^M_y+h^M_x+h^M_y}.
\end{eqnarray}

Temperley's bijection relates a dimer configuration $M$ of $G^d(v_0,f_0)$ to a spanning tree $T$ of $G$. The height function $h^M$ has a natural relation to the \textit{winding} of $T$, which is defined, for a finite directed path on $G$, as the total angle of the left turns minus the right turns along this path.

A \textit{branch} $\gamma=(v_1,...,v_{r+1})$ of an oriented tree $T$ is a finite directed path of $T$ keeping co-oriented or anti-oriented with the orientation of $T$ (either every edge $v_iv_{i+1}$ is oriented from $v_i$ to $v_{i+1}$ or from $v_{i+1}$ to $v_i$).

Denote the white vertex between $v_i$ and $v_{i+1}$ in $G^d$ by $w_i$. Let $(f_1,...,f_r)$ be the faces of $G^d$ lying on the left of $\gamma$ and every $f_i$ is incident to $v_i e_i$. Note that $f_i$ and $f_{i+1}$ are not neighboring in $G^d$. For any $i\in\{1,...,r\}$, define $\alpha^T(f_i)$ as the counterclockwise angle from the vector $v_i e_i$ to the diagonal of $f_i$. Note that for given $T$, $h^M(f_i)-\alpha^T(f_i)$ only depends on the vertex $v_i$ and doesn't depend on the choice of path (or face).

Theorem 3 in \cite{KPW} proves that under planar Temperley's bijection, winding of a branch $\gamma=(v_1,...,v_r)$ is equal to $\big(h^M(f_r)-\alpha^T(f_r)\big)-\big(h^M(f_1)-\alpha^T(f_1)\big)$.

To simplify notations, we define a height function on vertices of $G$: for any $T$ and $v\in T$, chose a branch passing $v$, define $h^T(v)$ as $h^M(f)-\alpha^T(f)$ where $f$ is a face incident to $v$ as above. So the theorem above says that going along a path, the change of height $h^T$ is equal to the winding.

On the torus, Temperley's bijection maps a dimer configuration $M$ of $\mathcal{G}^d_1$ to an $OCRSF$ pair $(F,F^*)$ of $\mathcal{G}_1$ and $\mathcal{G}^*_1$. Height change $(h^M_x,h^M_y)$ is closely related to the homology class of $(F,F^*)$. This fact is already shown by authors of \cite{Dub}. Here we give another proof because some geometric facts revealed in this proof are useful in the subsequent parts of this paper.

Suppose that $F$ has $k$ connected components, each component containing a root-cycle of homology class $\pm(m,n)$, $m,n\in\mathbb{Z}$, where we choose $m$ to be non-negative, and when $m=0$ we choose $n$ to be positive. Note that $m$ and $n$ are relatively prime. Suppose that there are $k_1$ (resp. $k_2$) primal (resp. dual) root-cycles of homology class $(m,n)$, then:

\begin{prop}(\cite{Dub})
The height change of $M$ can be expressed as a signed sum of homology classes of $OCRSF$ of ${\mathcal{G}}_1$ and of $\mathcal{G}^*_1$, as below:
$$h^M_x=-n(k-k_1-k_2),$$
$$h^M_y=m(k-k_1-k_2).$$
\end{prop}
\begin{proof} If $M$ is a dimer configuration of $\mathcal{G}^d_1$, then it gives rise to a $\mathbb{Z}^2$-periodic dimer configuration of $G^d$. Via Temperley's bijection, on the $\mathbb{Z}^2$-periodic graphs $G$ and $G^*$ this gives a pair of oriented spanning forest rooted on infinite paths. Each of its connected component is a tree, and we call a tree on $G$ as primal tree, a tree on $G^*$ as dual tree, and the only infinite path of a tree as its root.

Let us study the height change along the $y$ axis. Take a vertex $v_0$. We can choose a path on the infinite planar graph $G$ between $v_0$ and $v_1=v_0+\hat{y}$ in the following way.

On $G$ there are $mk$ primal trees $T$ and $mk$ dual trees $T^*$ between $v_0$ and $v_1$. Any edge on the root of a $T^*$ ($e$ in the figure) has two neighboring $G$-vertices ($v$ and $v'$ in the figure), lying on each side of the root. For both of them we follow the branches before we arrive at their roots. This gives a path between roots of two neighboring primal trees. We also allow walking along the roots of primal trees. Thus, by choosing one edge on every dual tree $T^*$, we construct a path on $G$ from $v_0$ to $v_1$ with $k$ jumps over the roots of dual trees.

\begin{figure}[H]
\centering
\subfigure[How to choose the neighboring $G$-vertices of $e$.]{
\includegraphics[width=0.35\textwidth]{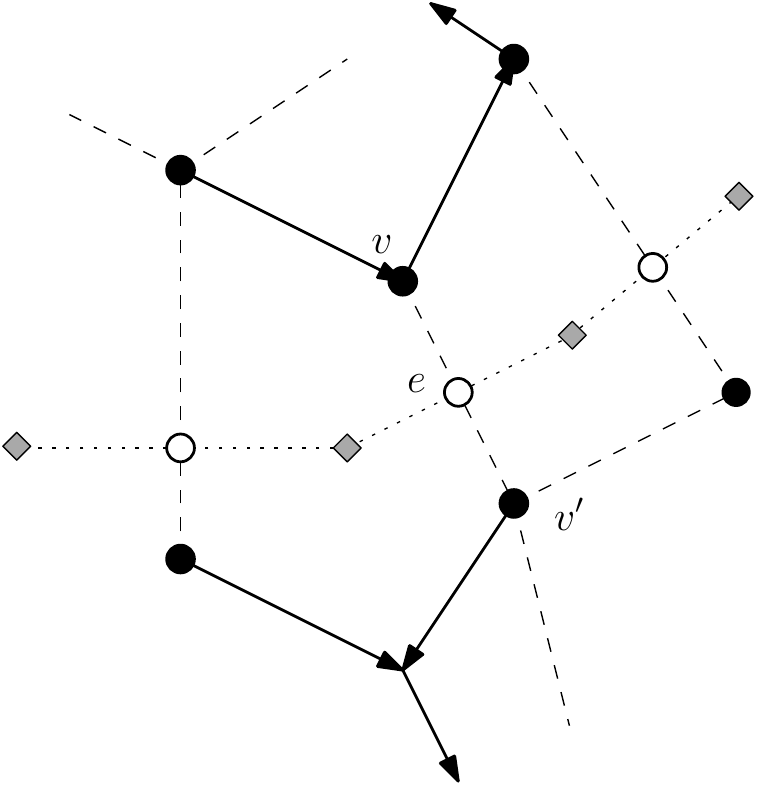}}
\subfigure[How to construct such path.]{
\includegraphics[width=0.55\textwidth]{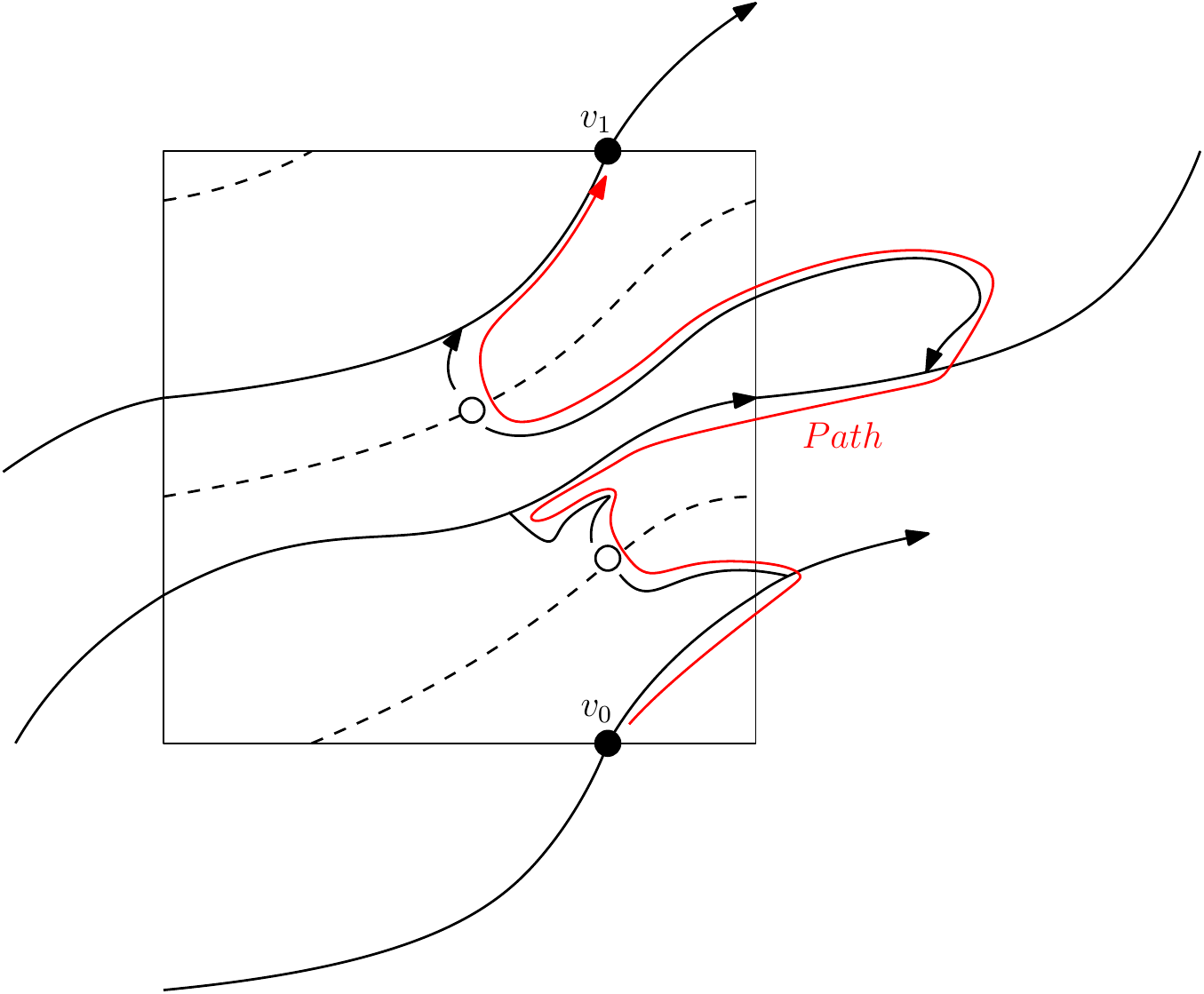}}
\caption{Construction of a periodic path.}
\end{figure}

When such a path jumps over a dual root, the local height change is the winding (by adding an imaginary edge between the two ends of the jump) plus $-\pi$ if it is from right to left over the root, and $\pi$ if it is from left to right. Walking along a path always co-oriented or anti-oriented gives a height change equal to winding. Entering a root, walking along the root and exiting into another branch, the observed orientation is reversed exactly once (from co-orientated to anti-oriented), either at the time of entering the root or the time of exiting the root. In both cases it can be viewed as joining another path and reversing the orientation. Joining from the right side of another path means a height change equal to winding plus $-\pi$ and from the left side means winding plus $\pi$. The proof is geometrical, as illustrated in the following figures.

After normalizing the height change by $2\pi$, we conclude that the total height change from $v_0$ to $v_1$ is the winding plus $\frac{1}{2}(-a+b)$ where $a$ is the number of the crosses over the roots of both primal and dual trees from right to left along the path, and respectively $b$ is the number of crosses from left to right. Since such path can be repeated between $v_0+n\hat{y}$ and $v_0+(n+1)\hat{y}$ for any $n\in\mathbb{Z}$ without self-joint, the winding of this path between $v_0$ and $v_1$ is $0$.

The height change along the $x$-axis is similar.

\qed
\end{proof}

\begin{figure}[H]
\centering
\subfigure[Jump: Winding=$(\alpha_1-\pi)+(\pi-\alpha_2)-(\pi-\alpha_3)=h+\pi$.]{
\includegraphics[width=0.35\textwidth]{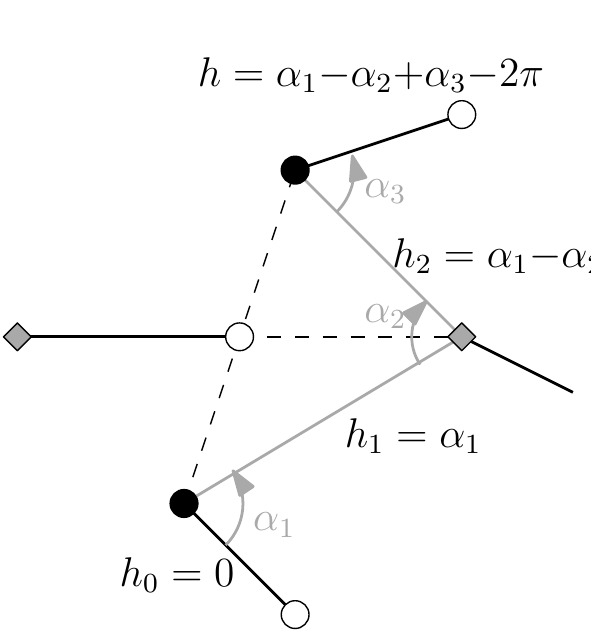}}
\subfigure[Join: Winding=$\pi-\alpha=h+\pi$.]{
\includegraphics[width=0.35\textwidth]{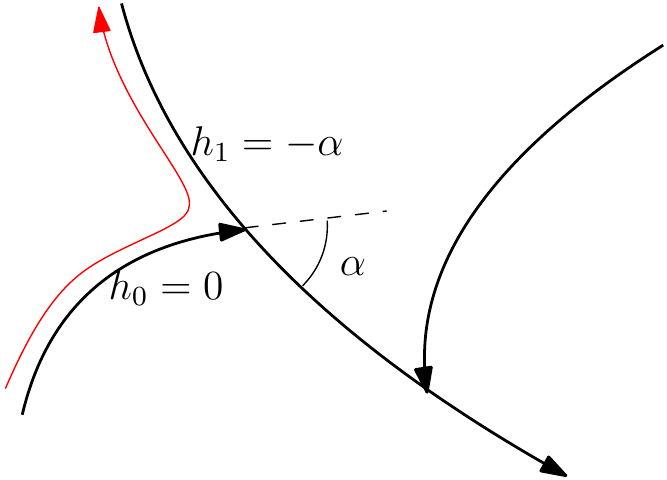}}
\caption{Jump or join from the right side of a directed path.}
\end{figure}

\section{Laplacian, Kasteleyn Matrix and measure}

A dimer probability measure on the $\mathbb{Z}^2$-periodic planar graph $G^d$ is characterized by the infinite inverse Kasteleyn matrix $K^{-1}$. By Temperley's bijection, the dimer measure gives rise to a measure on the directed edges of $G$. For the later one, it is more natural to characterize it by Laplacian, and if we take into account a magnetic field $B$, then we should consider the Laplacian with connection (Section 3.1). This characterization (Theorem 3.2) will be used in Section 5 to study the topology property of the configurations under the limiting measure.

\subsection{Laplacian with connection}

Following \cite{Ken11}, for a finite graph $G=(V,E)$, to every $v\in V$ and $e\in E$ we assign a space isomorphic to $\mathbb{C}$, denoted by $\mathbb{C}_v$ and $\mathbb{C}_e$. A \textit{connection} $\Phi$ on the graph $G$ is the choice for every edge $\vec{e}=uv$ an isomorphism $\phi_{uv}:\mathbb{C}_u\rightarrow \mathbb{C}_{v}$ such that $\phi_{uv}=\phi_{vu}^{-1}$. This isomorphism is called the \textit{parallel transport} from $\mathbb{C}_u$ to $\mathbb{C}_{v}$. This is generalized by assigning for every $\vec{e}=uv$ an isomorphism $\phi_{ve}:\mathbb{C}_v\rightarrow \mathbb{C}_e$ with the property that $\phi_{ve}=\phi_{ev}^{-1}$ and letting $\phi_{uv}=\phi_{ev}\circ\phi_{ue}$.

On a weighted and directed graph $G$, the \textit{Laplacian} associated to this connection $\Phi$ is the operator $\Delta^{\Phi}:\mathbb{C}^V\rightarrow\mathbb{C}^V$ defined by
$$\Delta^{\Phi}f(v)=\sum_{u\sim v}c_{vu}(f(v)-\phi_{uv}f(u)),$$
where the sum is over all neighbors of $v$.

The Laplacian can be decomposed in the following way. We fix an arbitrary edge orientation, then edges in $E$ can be viewed as directed edges. Let $\Lambda^0(G,\Phi)$ be the space of 0-forms and $\Lambda^1(G,\Phi)$ be that of 1-forms. Define $d:\Lambda^0(G,\Phi)\rightarrow\Lambda^1(G,\Phi)$ and $d^*:\Lambda^1(G,\Phi)\rightarrow\Lambda^0(G,\Phi)$ as
$$df(\vec{e})=\phi_{ve}f(v)-\phi_{v'e}f(v'),$$
$$d^*(\omega)(v)=\sum_{\vec{e}=v'v}c_{vv'}\phi_{ev}\omega(\vec{e}),$$
and we have the decomposition
\begin{eqnarray}
\Delta^{\Phi}=d^*d.
\end{eqnarray}

The Laplacian $\Delta^{\Phi}$ and the operators $d$ and $d^*$ can all be written in matrix form. The matrix $d^*$ has rows indexed by vertices of $G$ and columns indexed by edges of $G$ with chosen orientation, while $d$ has rows indexed by edges of chosen orientation and columns indexed by vertices. Note that the operator $d^*$ is a part of the Kasteleyn matrix, see Section 3.3.

In \cite{For}, the author proves that:
\begin{eqnarray}
\det\Delta^{\Phi}=\sum_{OCRSF}\prod_{\vec{e}\in F}c(\vec{e})\prod_{cycles\ \gamma}(1-w(\gamma)),
\end{eqnarray}
where the second product is the sum over all directed cycles $\gamma$ and $w$ is the monodromy of this cycle.

Equations (2.6), (3.8) and Proposition 2.1 yield the following proposition as a corollary, which says that the dimer characteristic polynomial of the double graph $\mathcal{G}^d$ arising from a toroidal graph $\mathcal{G}$ is also some Laplacian with connection. See also \cite{BdT08} where the authors prove this result for isoradial graphs.

\begin{prop}
On $\mathcal{G}$, choose two paths $\gamma_x$ and $\gamma_y$ on its dual graph, respectively winding once horizontally or vertically, and choose parallel transport as follows: $\phi_{vv'}=z$ (resp. $w$) if $vv'$ traverses $\gamma_x$ (resp. $\gamma_y$) from left to right, and $\phi_{vv'}=z^{-1}$ (resp. $w^{-1}$) if it traverses from right to left, otherwise let $\phi=1$, then
$$P_{dimer}(z,w)=\det\Delta^{\Phi}.$$
\end{prop}
\begin{proof}Since $m,n$ in Proposition 2.1 are relatively prime, they can not be both even. Let $a=(k-k_1-k_2)$, then the sign in (2.6) can be simplified as:
$$(-1)^{h_x h_y+h_x+h_y}=(-1)^{-nma^2-na+ma}=(-1)^a.$$
So,
\begin{eqnarray}
 P_{dimer}(z,w)
 &=&\sum_{M\in\mathcal{M}(\mathcal{G}^d)}\prod_{e\in M} c(e) z^{-h_x}w^{-h_y}(-1)^{h_x h_y+h_x+h_y}\nonumber \\
 &=&\sum_{(F,F^*)\in\mathcal{F}({\mathcal{G}},{\mathcal{G}}^*)}\prod_{\vec{e}\in F} c(\vec{e}) z^{-h_x}w^{-h_y}(-1)^{h_x h_y+h_x+h_y}\nonumber \\
 &=&\sum_{F\in OCRSF({\mathcal{G}})}\prod_{\vec{e}\in F} c(\vec{e})\sum_{k_2=0}^{k_2=k}
 {k\choose k_2}(-z^{n}w^{-m})^{k-k_1-k_2}\nonumber \\
 &=&\sum_{F\in OCRSF({\mathcal{G}})}\prod_{\vec{e}\in F} c(\vec{e})\cdot(1-z^{-n}w^{m})^{k_1}(1-z^{n}w^{-m})^{k-k_1}.
\end{eqnarray}
\qed
\end{proof}\\

\subsection{Magnetic field and connection}

Let $G$ be a $\mathbb{Z}^2-$periodic graph, $\mathcal{G}_N=G/(N\mathbb{Z})^2$. Following \cite{KOS06}, in the dimer model, by adding a \textit{magnetic field} $B=(B_x,B_y)$ on the toroidal bipartite graph $\mathcal{G}^d_N$, we mean choosing two dual paths $\gamma_x$ and $\gamma_{y}$ in $\mathcal{G}^d_1$ winding once horizontally or vertically around the torus, and for every edge of $\mathcal{G}^d_N$, if its copy in $\mathcal{G}^d_1$ cross $\gamma_{x}$ (resp. $\gamma_{y}$), multiplying its weight by $e^{\pm B_{x}}$ (resp. $e^{\pm B_{y}}$).

This is gauge-equivalent to choosing two dual paths $\gamma^N_x$ and $\gamma^N_y$ in $\mathcal{G}_N^d$ winding once horizontally or vertically around the torus, and letting $z=e^{NB_x}$ and $w=e^{NB_y}$ for $z$ and $w$ as in Proposition 3.1.

By Temperley's bijection, such modification gives rise to a modification of the weights of directed edges of $\mathcal{G}$ and $\mathcal{G}^*$: it changes the weight of a directed edge of $\mathcal{G}$ or $\mathcal{G}^*$ whose copy in $\mathcal{G}_1$ or $\mathcal{G}_1^*$ crosses $\gamma_x$ or $\gamma_y$ in one direction but not in the other, depending on the choice of $\gamma_x$ and $\gamma_y$ and the position of the edge. Proposition 3.1 says that an equivalent way to have this is just to consider the primal $OCRSF$ with corresponding parallel transport.

The dimer measure with magnetic field $B=(B_x,B_y)$ yields a natural probability measure on $OCRSF$ of $\mathcal{G}_N$. The partition function with this modification is a direct corollary of Proposition 3.1. It is equal to
$$\prod_{\vec{e}\in F} c(\vec{e})\cdot(1+e^{-nNB_x+mNB_y})^{k_1}(1+e^{nNB_x-mNB_y})^{k-k_1},$$
where the notations $m$, $n$, $k$ and $k_1$ are as in Proposition 3.1. In fact, if we replace all terms as $-z^{-n}w^{m}$ and $-z^{n}w^{-m}$  in (3.9) by $1$, the right hand side of (3.9) is
$$\sum_{F\in OCRSF({\mathcal{G}}_N)}2^k\prod_{\vec{e}\in F} c(\vec{e}),$$
which is the partition function without magnetic field.

\subsection{Laplacian and inverse of Kasteleyn matrix on finite graphs}

Equation (3.7), as is mentioned, can also be viewed as a matrix multiplication. Let $G$ be a finite graph, here we suppose that $G$ is planar or toroidal. For any given edge-orientation of $G$, the matrix $d^*$ is indexed by vertices of $G$ on rows and by edges of $\mathcal{G}$ with chosen orientation on columns, while $d$ is indexed by edges with chosen orientation on rows and by vertices on columns.

Any given edge-orientation of $G$ generates an orientation on edges of $G^*$, where the orientation of $\vec{e}^*$ is from the left side of $\vec{e}$ to its right side. If edges of the double graph $G^d$ inherit the orientation of those of $E$ and $E^*$, then its orientation is a Kasteleyn orientation. To see this, we remark that every simple face is a quadrilateral and we just need to verify the 4 possible cases.

We can also define the operators $d^*_{dual}$ and $d_{dual}$ as analog of $d^*$ and $d$ on the dual graph $G^*$.

If $G$ is toroidal, choose the connection corresponding to the magnetic field $B$ as in Proposition 3.1 and $(\theta,\tau)$ as in (2.2), and if $G$ is planar, we take a trivial connection. Then $\Delta^{\Phi}$ is the Laplacian with this connection, and the matrix
$$K=
\begin{pmatrix}
d^{*} \\
d^{*}_{dual}
\end{pmatrix}
$$
is a Kasteleyn matrix (with $z$ and $w$ in toroidal case) whose rows are indexed by vertices of $G$ and $G^*$ and whose columns can either be viewed as directed edges or as white vertices of $G^{d}$.

Similarly we define the matrix
$$
M=
\begin{pmatrix}
d & d_{dual}
\end{pmatrix}.
$$

Let $v\in\mathcal{G}$ and $v^*\in \mathcal{G}^*$ be opposite black vertices in any quadrilateral of $\mathcal{G}^d$, and denote the white vertices in this quadrilateral by $e_1$ and $e_2$. Without losing generality we suppose that dual edges are oriented form $v^*$ to $e_{1}$ and $e_2$, so the oriented primal edges are $e_1v$ and $ve_2$. Note that $c(v^*e_{1})=c(v^*e_2)=1$, we have
$$(d_{dual}^*d)_{v^*,v}=\phi_{e_1v^*}\phi_{ve_1}-\phi_{e_2v^*}\phi_{ve_2}=0,$$
and in other cases this value is trivially $0$.

Thus, we can write a matrix equation:
\begin{eqnarray}
K M
=
\begin{pmatrix}
d^* \\
d^*_{dual}
\end{pmatrix}
\begin{pmatrix}
d & d_{dual}
\end{pmatrix}
=
\begin{pmatrix}
\Delta^{\Phi} & \star \\
0 & \Delta^{\Phi}_{dual}
\end{pmatrix}.
\end{eqnarray}

By taking inverse of $K$ (when invertible), we have:
\begin{eqnarray}
K^{-1}
\begin{pmatrix}
\Delta^{\Phi} & \star\\
0 & \Delta^{\Phi}_{dual}
\end{pmatrix}
=M.
\end{eqnarray}

We will see that equation (3.11) gives a useful characterization of $K^{-1}$ for studying the limiting behavior of $OCRSF$.

\subsection{Infinite Laplacian and inverse of Kasteleyn matrix}

Now we focus on a $\mathbb{Z}^2$-periodic graph $G$ and its quotient graph $\mathcal{G}_N$. We denote the Kasteleyn matrix on $\mathcal{G}_N$ by $K_N$ (keep in mind that this is the Kasteleyn matrix corresponding to $B$ and $(\theta,\tau)$), and the one on $G$ by $K$. Same convention for other matrices.

In \cite{KOS06} the authors prove that when $N\rightarrow\infty$, $K^{-1}_N$ converge to a matrix $K^{-1}$ which is the inverse of $K$ (regardless of $(\theta,\tau)$ of $K^{-1}_N$). Equation (3.10) holds, and so does (3.11).

By construction of $K_N$, here the first half columns of $K^{-1}$ are indexed by vertices of $G$ and the second half are by those of $G^*$. We are only interested in the first half (primal $OCRSF$), fully described by $(K^{-1})^V$, where $(K^{-1})^V$ is the submatrix of the infinite matrix $K^{-1}$ whose columns are indexed by the vertices of $V$. We may write $K^{-1}=\left((K^{-1})^V\ (K^{-1})^{V^*} \right)$ and then by verifying the block product version of (3.11) we have $(K^{-1})^V\Delta^{\Phi}=d$.

Fixing a row in this equation means fixing some edge $e$ (\textit{i.e.} choosing a white vertex). Denote by $(K^{-1})^V_e$ and $d_e$ the corresponding row vectors of $(K^{-1})^V$ and $d$.

Theorem 3.2 below gives a description of $(K^{-1})^V$ by a statement of existence and uniqueness. A similar argument can be found in \cite{BdT08}.

A vertex $u$ of $G^d$ can be written in the form $(x,y;v)$, where $(x,y)\in\mathbb{Z}^2$, $v\in\mathcal{G}^d_{1}$. Define $\mathcal{C}_0(\mathbb{Z}^2)$ as the space of $\mathcal{G}^d_{1}$-vector-valued functions decaying at infinity, and define $\mathcal{C}_0^B(\mathbb{Z}^2)$ as its (magnetic field $B$) modified version:
$$\mathcal{C}_0^B(\mathbb{Z}^2):=\{f:\mathbb{Z}^2\rightarrow\mathcal{G}^d_{1}: e^{xB_y+yB_x}f(x,y;v)\in \mathcal{C}_0(\mathbb{Z}^2)\}.$$

\begin{theorem}
The matrix $(K^{-1})^V$ is the unique infinite matrix $A$ such that every row $A_e\in \mathcal{C}_0^B(\mathbb{Z}^2)$ and $A\Delta^{\Phi}=d$.
\end{theorem}
\begin{proof} To prove the uniqueness here we use Fourier transform. Following \cite{BdT08}, the space of rapidly decaying $\mathcal{G}^d_{1}$-vector-valued functions is
$$\mathcal{S}(\mathbb{Z}^2):=\{f:\mathbb{Z}^2\rightarrow\mathcal{G}^d_{1}: \forall(m,n)\in\mathbb{Z}^2, \lim_{\|(x,y)\|\rightarrow\infty}\|x^my^nf(x,y)\|=0\},$$
and its $B$-modified version:
$$\mathcal{S}^B(\mathbb{Z}^2):=\{f:\mathbb{Z}^2\rightarrow\mathcal{G}^d_{1}: e^{xB_y+yB_x}f(x,y;v)\in \mathcal{S}(\mathbb{Z}^2)\}.$$
Also denote by $\mathcal{S}(\mathbb{T}^2)$ the space of $\mathcal{G}^d_{1}$-vector-valued smooth function on the torus.

The Fourier transform of a $\mathcal{G}^d_{1}$-vector-valued function $f$, when exists, is
$$\widehat{f}(z,w)=\sum_{(x,y)\in\mathbb{Z}^2}f(x,y) w^x z^y,\ (z,w)\in\mathbb{T}^2,$$
and we define the Fourier transform with magnetic field $B$ as
$$\widehat{f}^B(z,w)=\sum_{(x,y)\in\mathbb{Z}^2}f(x,y) (we^{B_y})^x (ze^{B_x})^y,\ (z,w)\in\mathbb{T}^2.$$
The fourier transform gives a bijection between $\mathcal{S}(\mathbb{Z}^2)$ and $\mathcal{S}(\mathbb{T}^2)$. Denote by $\langle\ ,\ \rangle_{\mathbb{Z}^2}$ (resp. $\langle\ ,\ \rangle_{\mathbb{T}^2}$) the duality bracket between between $\mathcal{S}(\mathbb{Z}^2)$ and its dual $\mathcal{S'}(\mathbb{Z}^2)$ (resp. between $\mathcal{S}(\mathbb{T}^2)$ and its dual $\mathcal{S'}(\mathbb{T}^2)$). The Fourier transform extends as a bijection from $\mathcal{S}'(\mathbb{Z}^2)$ to $\mathcal{S}'(\mathbb{T}^2)$ by duality.

The Laplacian acting on the right side is:
$$f\Delta^{\Phi}(u)=\sum_{u'\sim u}c_{u'u}[f(u)-\phi_{uu'}f(u')],$$
where the parallel transport is
$$\phi_{u,u'}=\phi_{(x,y;v)(x',y';v')}=e^{B_y(x'-x)+B_x(y'-y)}.$$

Thus,
\begin{align*}
\widehat{f\Delta^{\Phi}}^B &=\sum_{(x,y)\in\mathbb{Z}^2}(we^{B_y})^x(ze^{B_x})^y\sum_{u'\sim u}c_{u'u}[f(u)-e^{B_y(x'-x)+B_x(y'-y)}f(u')]\\
&=\sum_{(x,y)\in\mathbb{Z}^2}w^xz^y\sum_{u'\sim u}c_{u'u}[e^{xB_y+yB_x}f(u)-e^{x'B_y+y'B_x}f(u')]\\
&=\widehat{g_f\Delta},
\end{align*}
where $g_f(x,y,v)=e^{xB_y+yB_x}f(x,y,v)$ and $\Delta$ is the Laplacian with trivial connection. By definition we see that when $f\in \mathcal{C}_0^B(\mathbb{Z}^2)$, then $g_f\in \mathcal{C}_0(\mathbb{Z}^2)$, and the action of $\Delta$ preserves the space $\mathcal{C}_0(\mathbb{Z}^2)$.

To prove the uniqueness it suffices to show that the only solution of $f\Delta^{\Phi}=0$ in $\mathcal{C}_0^B(\mathbb{Z}^2)$ is $0$. Then its Fourier transform with $B$, which is equal to $\widehat{g_f\Delta}$, is also $0$. Their Fourier transforms (with or without $B$) are well defined, and for any test function $h\in \mathcal{S}(\mathbb{T}^2)$,
\begin{align*}
0 &= \langle \widehat{g_f\Delta},h\rangle_{\mathbb{T}^2}\\
&= \langle g_f\Delta,\check{h}\rangle_{\mathbb{Z}^2}\\
&= \langle g_f, \Delta\check{h}\rangle_{\mathbb{Z}^2}\\
&= \langle \widehat{g_f}, \widehat{\Delta}h \rangle_{\mathbb{T}^2}.
\end{align*}

The second line and fourth line are by Parseval's theorem, the fourth line is also by the fact that $\Delta$ acts on $\check{h}$ as a convolution rather than a product. The third line is well defined as $g_f\in\mathcal{C}_0(\mathbb{Z}^2)\subset\mathcal{S}'(\mathbb{Z}^2)$, and $\widehat{g_f}$ in the forth line is in $\mathcal{S}'(\mathbb{T}^2)$ defined by duality. Since $\widehat{\Delta}$ is invertible except at $(1,1)$, the above calculations show that $\forall\psi\in\mathcal{S}(\mathbb{T}^2)$ s.t. $\hat{\psi}$ has support contained in $\mathbb{T}^2\backslash\{(1,1)\}$, let $h$ be $\widehat{\Delta}^{-1}\psi\in\mathcal{S}(\mathbb{T}^2)$, so the support of $\widehat{g_f}$ is contained in $\{(1,1)\}$. For $g_f\in\mathcal{C}_0(\mathbb{Z}^2)$, the only possibility is $g_f=0$, so $f=0$.

To prove the existence, knowing that $(K^{-1})^V$ exists and verifies $(K^{-1})^V\Delta^{\Phi}=d$, we should also prove that every row $(K^{-1})^V_e$ is in the space $\mathcal{C}_0^B(\mathbb{Z}^2)$. By definition it is equivalent to proving that $g_{(K^{-1})^V_e}\in\mathcal{C}_0(\mathbb{Z}^2)$, and we have
$$\widehat{g_{(K^{-1})^V_e}\Delta}=\widehat{(K^{-1})^V_e\Delta^{\Phi}}^B=\widehat{d_e}^B=\widehat{g_{d_e}}.$$
In the last term $g_{d_e}=e^{B_yx+B_xy}d^0_e$, $d^0_e$ is the matrix form of $d$ where the magnetic field
$B$ is $0$. Thus, in the case that $e$ satisfies $x(e)=y(e)=0$, $(K^{-1})^V_e\in \mathcal{C}_0(\mathbb{Z}^2)$ by the same proof as in Proposition 5 of \cite{BdT08}, where the crucial fact is that Proposition 3.1 gives a characterization of the zeros of $\det K(z,w)$ on the torus $\mathbb{T}^2$. By translation invariance it is true for all $e$.

\qed
\end{proof}

\subsection{Measures on $\mathcal{G}_N$ and $G$}

A dimer probability measure on the $\mathbb{Z}^2$-periodic planar graph $G^d$ can be obtained as a limit of Boltzmann probability measures on $\mathcal{G}^d_N$ when $N$ goes to infinity. In \cite{KOS06}, the authors prove that the limiting measure $\mu$ is a determinantal process with kernel $K^{-1}$.

For all $N\in\mathbb{N}$, Temperley's bijection gives a probability measure on $OCRSF$ of $\mathcal{G}_N$. The results of \cite{KOS06} and Temperley's bijection directly imply that, when $N\rightarrow\infty$, $OCRSF$ measures also converge weakly to a limiting Gibbs measure $\mu$ (we use the same letter since they are the same measure) on the configurations of the directed edges of the $\mathbb{Z}^2$-periodic planar graph $G$. This measure is a determinantal process with kernel $(K^{-1})^V$, which is characterized by Theorem 3.2.

Now we give a brief discussion on the measures on non-oriented edges. As mentioned in the introduction, the main interest of this paper is the topology of the configurations under the limiting measure. The similarity of the non-oriented-edge measure and the spanning-tree measure inspires our Section 5.

Repeat the result of \cite{KOS06} in the language of $OCRSF$. Oriented edges $\vec{e}_i$ in $OCRSF$ form a determinantal process, and
\begin{eqnarray}
\mathcal{P}(\vec{e}_1,...,\vec{e}_m)=\big(\prod_i K(\vec{e}_i)\big)\det(K^{-1})_{\vec{E}},
\end{eqnarray}
where $\vec{E}$ is the set of oriented edges $\{\vec{e}_1,...,\vec{e}_m\}$, $K(\vec{e}_i)$ is equal to $\pm c(\vec{e}_i)\phi_{\vec{e}}$ where its sign depends on the orientation and $\phi_{\vec{e}}$ is the parallel transport along $\vec{e}$. We write $\vec{e_i}=v_i v'_i$. Rows of $(K^{-1})_{\vec{E}}$ are indexed by edges of $\vec{E}$ and columns are indexed by $\{ v_i \}$, the starting points of the oriented edges $\{\vec{e_i}\}$.

We suppose that edges of $\vec{E}$ have no common edges and no common starting points, otherwise this is automatically $0$. By entering $K(\vec{e}_i)$ into columns, we can rewrite the term on the right hand side of (3.12) as a single determinant. The $(e_i,v_j)^{th}$ element of this matrix is
$$K(\vec{e}_j)\left( K^{-1}_{e_iv_j} \right).$$

Consider the probability of non-directed edges $E$, which is a binomial sum over directed edges. Denote the reverse of $\vec{e}$ by $\check{\vec{e}}$. Such probability measure is a determinantal process with an edge-edge matrix kernel whose $(e_i,e_j)^{th}$ element is
\begin{eqnarray}
K(\vec{e}_j)\left( K^{-1}_{e_iv_j} \right)+
K(\check{\vec{e}}_j)\left( K^{-1}_{e_iv_j'} \right).
\end{eqnarray}

We note that when the magnetic field is $0$ and the weights of edges are all equal to $1$, formally (3.11) says that $K^{-1}$ is the difference of two Green's functions $g$, and we can rewrite (3.13) as
$$\left(g(v_i',v_j)-g(v_i,v_j)\right)-
\left(g(v_i',v_j')-g(v_i,v_j')\right)
,$$
although only differences of $g$ make sense.

This result is the same as the probability measure on uniform spanning trees on the infinite equal-weighted $\mathbb{Z}^2$-lattice studied in \cite{BP}. Thus for $G=\mathbb{Z}^2$ with equal weight-setting, $OCRSF$ of $\mathcal{G}_N$ converges weakly to spanning trees of $\mathbb{Z}^2$.

\section{Non-zero slope}

In the previous section, by comparing to the results of \cite{BP}, we showed that in the simplest case (square lattice whose edges are equally weighted), under the limiting measure there is a.s. exactly one connected component. A natural question is about the number of connected components in more general cases. In this section, we will prove that when the slope of the limiting measure is non-zero, then there are a.s. infinitely many connected components, and in the next section we prove that in some setting, zero slope means one connected component.

Consider a $\mathbb{Z}^2$-periodic planar graph $G$. We may add a magnetic field on its double graph $G_d$. Following \cite{Pem}, for two given vertices $v_1$ and $v_2$ in $G$, and for any $N\in\mathbb{N}$ such that $\mathcal{G}_N$ contains $v_1$ and $v_2$, we consider the event in $\mathcal{G}_N$ that $v_1$ and $v_2$ are connected within a ball $B_L$ by $OCRSF$ of the toroidal graph $\mathcal{G}_N$ (we ask that the size $N$ of torus is larger than the diameter of $B_L$ so that $B_L$ doesn't superpose with itself). The probability that, under the limiting measure $\mu$, $v_1$ and $v_2$ are connected is equal to:
$$\lim_{L\rightarrow}\lim_{N\rightarrow\infty}\mathbb{P}_{\mathcal{G}_N}[v_1\text{ and }v_2\text{ are connected within }B_L].$$

The measure on $OCRSF$-pairs of the torus gives rise to a measure on their roots (oriented cycles). Proposition 2.1 proves that for any simple close curve $\gamma_{x}$ (resp. $\gamma_y$) that winds once horizontally (resp. vertically), the signed sum of the crossings of the oriented cycles on such a curve is equal to the horizontal (resp. vertical) height change.

\begin{lemma}
For $(F,F^*)$ as an $OCRSF$ pair, we omit the branches and look at the cycles. Vertices $v_1$ and $v_2$ are not connected within $B_L$ if the absolute value of the signed sum of the number of the cycles passing $\gamma$ between $v_1$ and $v_2$ is not less than two.
\end{lemma}
\begin{proof}For any simply connected finite region on the torus, if two vertices lie on different side of a dual cycle, then they are not connected within the region. If the signed sum of the cycles passing $\gamma$ between these vertices is not less than two, then there should be at least one dual cycle passing $\gamma$ between these vertices.
\qed
\end{proof}

When $N\rightarrow\infty$, the average height change converges to the slope $(s,t)$ of the Gibbs measure $\mu$, see \cite{KOS06}.

\begin{theorem}
If the slope $(s,t)$ of the limiting dimer measure is non-zero, then under the limiting $OCRSF$ measure there are a.s. infinitely many connected components.
\end{theorem}
\begin{proof}We suppose that $t>0$. On $G$ we choose $v_1$ arbitrarily and choose $v_2$ being a copy of $v_1$ laying $k$ units upper ($v_2=v_1+k\hat{y}$).

We choose $\gamma_y$ as a periodic path on $G$ (so on the path there are only the white vertices and the primal black vertices of $G^d$), winding vertically and passing $v_1$ and $v_2$.  Let $M-1$ be the number of the black vertices on $\gamma_y$ between $v_1$ and $v_1+\hat{y}$, so there are $kM-1$ black vertices and $kM$ white vertices on $\gamma_y$ between $v_1$ and $v_2$.

For any ball $B_L$ that contains $v_1$ and $v_2$, and for any $N$ large enough such that $B_L$ is contained in $G_N$ as a simply connected set, Lemma 4.1 says that the probability that $v_1$ and $v_2$ are not connected are bounded from below by the probability that the signed sum of the cycles passing through $\gamma$ between them are strictly bigger than $1$ or strictly less than $-1$.

For any $N$, and any $OCRSF$ of $\mathcal{G}_N$, the total height change along $y$-axis $h_y$ is the total signed sum of the primal cycles and dual cycles that pass through the once-vertically winding curve $\gamma_y$. So for $v_1$ and $v_2$, the expected height difference $h_y(v_1,v_2)$ is the expected signed sum of the number of crossings of the cycles between $v_1$ and $v_2$.

As $N$ goes to infinity, $\mathbb{E}[h_y(v_1,v_2)]\rightarrow kt$. So $\forall\alpha>0$ small, when $N$ is large, we have $\mathbb{E}[h_y(v_1,v_2)]\geq kt-\alpha$. As there are at most $2kM-1$ vacancies that allow the cycles to pass through, the expectation of the signed sum can be written as:
$$p_{-2kM+1}(-2kM+1)+...+p_{-1}(-1)+p_1+...+p_{2kM-1}(2kM-1),$$
where $p_i$ is the probability that the signed sum is $i$. And if $|i|\geq 2$, then there are at least two cycles passing between $v_1$ and $v_2$.

In hypotheses we suppose that $t$ is positive. We want to maximize $p_{-1}+p_{0}+p_{1}$ under the constraints $$\sum_{i=-2kM+1}^{2kM-1}p_{i}i\geq kt-\alpha,\ \sum_{i=-2kM+1}^{2kM-1}p_{i}=1.$$

If $p_{-1}+p_{0}+p_{1}$ is equal to $p$, then their contribution to the expectation is at most $p$, and the remaining terms contribute at most $(1-p)(2kM-1)$. So we have
$$p+(1-p)(2kM-1)\geq kt-\alpha,$$
which turns to be
$$p(2kM-2)\leq 2kM-kt-1+\alpha.$$

Choose $k$ bigger than $1/M$ and $1/t$, $\alpha$ sufficiently small, then $p<\frac{2kM-1-kt+\alpha}{2kM-2}$, which is less than $1$, and this is bigger than the probability that $v_1$ and $v_2$ is not connected. Especially, we remark that this upper bound when $N$ is large enough doesn't depend on $L$. As $L$ tends to infinity, the probability that $v_1$ and $v_2$ are connected is less than $1$, so the probability that there is a unique connected component is less than $1$.

By the same method we can generalize this result to finite vertices, saying that there exist $v_1,...v_k\in V$, the probability that any two of them is connected is less than $1$, so the probability that there are at most $k-1$ connected components is less than $1$.

Since the measure is ergodic, and whether there are at most $k-1$ components is a translation-invariant event, the probability that there are at most $k-1$ connected components is $0$, and we prove the theorem.
\qed
\end{proof}

\section{Zero slope}
The following lemma is an important observation for a graph $G^d$ arising form its primal graph $G$.

\begin{lemma}
In the phase diagram of the dimer measure of $G^d$, the point $B=(0,0)$ always corresponds to a zero slope.
\end{lemma}
\begin{proof}
When $B=(0,0)$, $(z,w)=(e^0,e^0)=(1,1)$ is always a real zero of the characteristic polynomial $\det K(z,w)=\det \Delta^{\Phi}$, so either $B=(0,0)$ lies on the boundary of the amoeba (when $(1,1)$ is a single root) or in the interior of the amoeba (when $(1,1)$ is a double root). In either case, it corresponds to an integer point in the Newton's polygon.

If the graph $G$ have a symmetric weight setting (\textit{i.e.} $c(uv)=c(vu)$ for all edges $uv$), then the Laplacian $\Delta^{\Phi}$ is symmetric in $z$ and $z^{-1}$ (resp. in $w$ and $w^{-1}$). Since $\det K(z,w)=\det \Delta^{\Phi}$, the amoeba is symmetric with respect to the origin and so does the Newton's polygon, and $B=(0,0)$ corresponds to $(s,t)=(0,0)$. For this symmetric Laplacian, $(1,1)$ is a double real root, so $B=(0,0)$ lies in the interior of the amoeba (a liquid phase). Also, as an interior integer point, $(s,t)=(0,0)$ corresponds either to a liquid phase or to a gaseous phase.

In general, any weight setting can be obtained from the symmetric weight setting via continuous deformation. Along this deformation, for any fixed magnetic field $B$ the slope $(s,t)$ changes continuously, while the point $B=(0,0)$ always corresponds to an integer point $(s,t)$. Thus $B=(0,0)$ always corresponds to $(s,t)=(0,0)$. This finishes the proof.

\qed
\end{proof}

Note that same slope means same limiting measure. To study the case where the slope is zero, we just need to study the case where the magnetic field is zero. The advantage is that $B=(0,0)$ enables us to approach the limiting Gibbs measure $\mu$ by another sequence of measures $m_N$ but on finite planar graphs. The later one has a random-walk interpretation, which gives some tools to study connectivity.

In the following part we suppose that $B=(0,0)$, and we omit the connection $\Phi$ (which is trivial) of the Laplacian $\Delta^{\Phi}$ to simplify the notation.

Let the finite graphs $(G_N)_N$ form an exhausting sequence of $G$ but with a \textit{wired boundary condition}. A wired boundary is to glue every vertex on the boundary into one. Let $G^*_N$ be its dual and let $G^d_N$ be the double graph. They are both square lattice except vertices near boundary.

For a non-oriented spanning tree of $G_N$, we choose the boundary vertex $r$ as root, denote the tree oriented to $r$ by $T$, and choose an arbitrary vertex in $G^*_N$ incident to $r$, denoted by $r^*$. $T^*$ is dual of $T$ rooted at $r^*$. Such $T$ is called \textit{wired spanning trees} ($WST$), which was implicit in \cite{Pem}, then made explicit in \cite{Hagg95} and further developed in \cite{BLPS}. We denote this weighted wired spanning tree measure by $m_N$.

In \cite{BLPS}, the authors prove the existence of a weak limit measure on $WST$ on non-directed weighted graphs, also called \textit{networks}. Such name is given because of its natural relation to electrical networks.

Graphs arising from dimer models are directed and weighted. We show that the same approach still works.

Recall that for the finite planar graph $G_N$ with wired boundary condition, Temperley's bijection gives a measure preserving bijection between dimer configurations of $G^d_N\setminus\{r,r^*\}$ and spanning-tree-pairs $(T,T^*)$, $T$ rooted at $r$ and $T^*$ rooted at $r^*$, see Section 2.2.1. Here the weight of $T^*$ is always $1$ and every $T$ has only one dual $T^*$, so the spanning-tree-pair measure is the same as $m_N$.

A spanning tree of a finite planar graph is automatically related to a random walk by Wilson's algorithm.

A \textit{random walk} on weighted graph $G$ is a Markov chain $X_0,X_1,X_2,...$ on $G$ that for all $n\in\mathbb{Z}$, $v,u\in G$,
$$\mathbb{P}(X_{n+1}=v|X_{n}=u)=\frac{c(uv)}{\sum_{v'\sim u}c(uv')}.$$

Let $p=(v_0,v_1,...,v_n)$ be a path on $G$. The \textit{loop erasure} of $p$ is a path $LE(p)=(u_0,u_1,...,u_m)$, such that $u_0=v_0$, and conditioned on that $u_j$ is set and $k$ is the largest number such that $v_k=u_j$, then $u_{j+1}=v_{k+1}$.

\textit{Wilson's algorithm} (\cite{Wilson},\cite{BLPS}): for the finite graph $G_N$ and root $r$ chosen as a vertex of $G_N$, the algorithm constructs a growing sequence of trees $\big(T(i)\big)_i$ from $T(0)=r$, and once $\big(T(i)\big)_i$ is generated, we randomly and independently pick a vertex $v$ not in $T(i)$, start an independent random walk $X$ starting at $v$ and end it once it hits $T(i)$. The new tree $T(i+1)$ is defined as $T(i)$ plus the loop erasure of the path of $X$. Continue this process until every vertex is in the tree. The constructed spanning tree has a probability proportional to its weight.

We forget the restriction of staying in $G_N$ and consider random walk in the $\mathbb{Z}^2$-periodic graph $G$. For any $N$, consider $G_N$. Similar to what we have in toroidal case, the oriented-edge-measure of the spanning trees of $G_N$ rooted at $r$ forms a determinantal measure of kernel $(K^{-1}_N)^V$, which is the submatrix of the inverse Kasteleyn matrix of $G^d_N\setminus\{r,r^*\}$ indexed by vertices of $G_N$.

Here the matrix relation (3.7) is $\Delta_N=d^*_N d_N$. We rearrange the columns and rows such that first half of the blacks are indexed by the vertices of $G_N$. Taking away $r$ and $r^*$ corresponds to deleting the corresponding rows and columns in the matrices, the matrices such modified are denoted by symbols with tildes:
$$
\begin{pmatrix}
\widetilde{d}^{*}_N \\
\widetilde{d}^{*}_{dual\ N}
\end{pmatrix}
\begin{pmatrix}
\widetilde{d}_N & \widetilde{d}_{dual\ N}
\end{pmatrix}
=
\begin{pmatrix}
\widetilde{\Delta}_N & \star \\
0 & \widetilde{\Delta}_{dual\ N}
\end{pmatrix}
.
$$

Removing $r$ and $r^*$ leaves $\widetilde{\Delta}_N$ and $\widetilde{\Delta}_{dual\ N}$ invertible. The matrix $\begin{pmatrix}\tilde{d}^*_N \\ \tilde{d}^*_{dual\ N} \end{pmatrix}$ is exactly the Kasteleyn matrix $K_N$. So $(K^{-1}_N)^V$ is the only matrix $A_N$ which satisfies
$$A_N \widetilde{\Delta}_N=\widetilde{d}_N.$$

Let $D$ be a diagonal matrix indexed by $v\in G_N\setminus\{r\}$, and $D_{v,v}=\sum_{v'\sim v}c(vv')$. Then
\begin{equation*}
\left(D^{-1}_N \widetilde{\Delta}_N\right)_{v,v'}=
\begin{cases}
1 & v'=v\\
-c(vv')/\sum_{v''\sim v}c(vv'') & v'\neq v.
\end{cases}
\end{equation*}

When $v'\neq v$, the $(v,v')^{th}$ entry is the transition probability of the random walk from $v$ to $v'$. Write
$$p_{v,v'}=c(vv')/\sum_{v''\sim v}c(vv'').$$

The only matrix $B_N$ satisfying $B_N D^{-1}_N\widetilde{\Delta}_N=\widetilde{d}_N$ is $(K^{-1}_N)^VD_N$. Meanwhile, there is a natural solution of this equation given by Green's function. For white vertex $w$ associated to directed edge $\vec{e}=v_1v_2$ (see Figure 11), and for any vertex $v$, if we define $RW_{a}^{G_N}$ as the random walk on $G_N$ starting at $a$ and killed at boundary $r$, then
\begin{eqnarray}
(B_N)_{w,v}=\mathbb{E}\left[\# RW_{v_2}^{G_N}\text{ visits }v-\# RW_{v_1}^{G_N}\text{ visits }v\right].
\end{eqnarray}

\begin{figure}[H]
\centering
\includegraphics[width=0.22\textwidth]{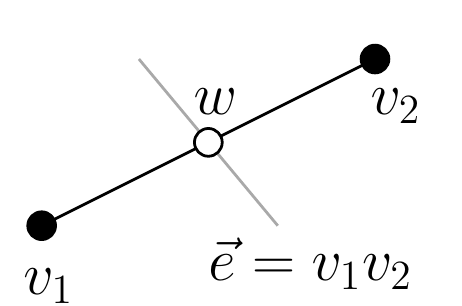}
\caption{$\vec{e}=v_1v_2$.}
\end{figure}

\begin{condition} If the right hand side of (5.14) converges when $N\rightarrow\infty$ and decays to zero when the distance between $w$ and $v$ tends to infinity, we say that the graph $G$ verifies the condition ($\star$).
\end{condition}
\begin{theorem}
When the condition ($\star$) is verified, as $N$ goes to infinity, $m_N$ converges to a measure $m$ on the drifted square grid graph $G$ which is the same as $\mu$, the weak limit of $\mu_N$. The measure $m$ is supported on spanning trees, so is $\mu$.
\end{theorem}

The condition ($\star$) is true for a big class of graphs. See the following propositions.

\begin{prop}
If the graph $G$ is transient, then the condition ($\star$) is verified.
\end{prop}
\begin{proof}
We have
$$\mathbb{E}\left[\# RW_{v_2}^{G_N}\text{ visits }v\right]=\mathbb{P}\left(RW_{v_2}^{G_N}\text{ visits }v\right)\mathbb{E}\left[\# RW_{v}^{G_N}\text{ visits }v\right].$$
In a transient $\mathbb{Z}^2$-periodic case, the second factor on the right hand side converges when $N\rightarrow\infty$ and is bounded. The first factor also converges when $N\rightarrow\infty$, and it tends to zero when the distance between $v_2$ and $v$ tends to infinity, since in scaling the random walk $\gamma$ behaves like a drift of order $n$ plus a term of variance $\sqrt{n}$. For a fixed-size ball, the probability that such a path visits it decays to zero as the distance between the ball and the origin tends to infinity.
\qed
\end{proof}

\begin{prop}
If the graph is non-directed, then the condition ($\star$) is verified.
\end{prop}
Non directed graph is a network, and such property can be seen in \cite{BLPS}.

\begin{example} The drifted grid graph.
\end{example}
Here we look at an example: the drifted square grid graph, which is a square lattice with drifted weight setting: every vertex have four incident edges with the conductances being $a$, $b$, $c$ and $d$ clockwise, see Figure 12. Its dual $G^*$ is the square lattice with edges weighted $1$. The fundamental domain of $G^d$ is the same as the example in Figure 4.

\begin{figure}[H]
\centering
\includegraphics[width=0.25\textwidth]{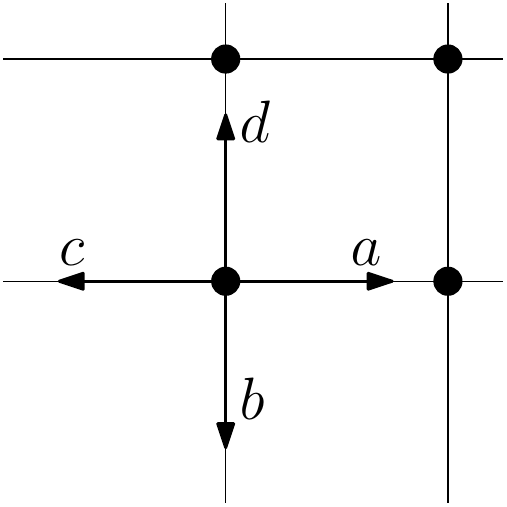}
\caption{Drifted square grid.}
\end{figure}

The phase diagram of the dimer model on $G^d$ with typical weighting is as in Figure 1 (black and grey curves give the amoeba). There is only one possible bounded gaseous region for any value of $(a,b,c,d)$ such that $a\neq c$ or $b\neq d$, otherwise such region vanishes. In Figure 1 the gaseous region is in light blue.

If the random walk associated is recurrent, then $a=c$ and $b=d$. This is a network, so Proposition 5.5 applies. Otherwise this is a transient graph and Proposition 5.4 applies. So the drifted square grid graph always satisfies the condition ($\star$).

\bigskip
\noindent
\textit{Proof of Theorem 5.3. }
When the condition ($\star$) is verified, any entry of $B_N$ converges when $N\rightarrow\infty$. Denote its limit by $B$. Measures $(m_N)_N$ converge to a limiting measure $m$. The entry of $B$ decays to $0$ as the distance between two vertices tends to infinity. Let $A=BD^{-1}$. It is the kernel of $m$ and satisfies the equation $A\Delta=d$. Each row vector of the matrix $A$ can be viewed as a function on the vertices of $G$.

Theorem 3.2 says that the kernel of $\mu$ is the unique matrix verifying this equation and decays at infinity. This proves that $m=\mu$. To finish the proof of Theorem 5.3, we just need to prove that the measure $m$ is supported by spanning trees, and this is proven in Lemma 5.7.

\begin{lemma}
The measure $m$ is supported on spanning trees of $\mathbb{Z}^2$.
\end{lemma}
\begin{proof} In \cite{Pem}, the author shows that spanning trees of equal weighted square grid converge to trees of $\mathbb{Z}^2$ if and only if independent simple random walk and loop erased random walk intersect infinitely often a.s. The same argument still applies to other cases. This is also known to the authors of \cite{LPS} in their Proposition 2.1. Here in our case where the weight function is defined on directed edges, there is nothing new.

Theorem 1.1 in \cite{LPS} shows that, for two independent transient Markov chains $RW^1$ and $RW^2$ on the same graph and having the same transition probabilities, if the path of $RW^1$ and that of $RW^2$ intersect infinitely a.s., then $LERW^1$ and $RW^2$ intersect infinitely a.s. too.

In our case, in scaling the random walk $\gamma$ behaves like a drift term of order $n$ plus a term of variance $\sqrt{n}$. The paths of two independent random walk meets a.s. as the time tends to infinity. This finishes the proof.
\qed
\end{proof}

\noindent
\textbf{Remark:} Here we choose wired spanning tree measure $m$ to approach $\mu$. However, we conjecture that the local behavior of the spanning tree finally does not depend on the choice of root $r$ on the boundary of $G_N$, \textit{i.e.} we choose the root vertex simply to be a vertex on the boundary of $G_N$ instead of gluing the boundary, and when $N\rightarrow\infty$ this always converges to the same measure no matter where the root is.

Our result is true for any graph with the property ($\star$), among which the drifted square grid graph is an interesting example. Proposition 5.4 works for all transient graphs. So the main difficulty for getting such results as Theorem 5.3 on general $\mathbb{Z}^2$-periodic graphs is that we don't know how to prove that the difference of the Green's function for recurrent random walk on directed graphs killed at wired boundary converges when the size of graph tends to infinity and decays when the distance of the vertices tends to infinity. We conjecture that this is true, and we remark that without this boundary condition, the decay of the difference of the Green's function can be found in some references, for example, \cite{KTU}.

\section{Phase diagram}

Combining the results in Section 4 and Section 5, we give a phase diagram for graphs verifying the condition ($\star$). Figure 1 gives the phase diagram of a drifted square grid graph of a typical weighting (Figure 12). If for this $G$ we take $\mathcal{G}_2$ as its fundamental domain and for each of its four vertices we independently assign an arbitrary weighting, the phase diagram is as Figure 13.

The bounded closed set corresponding to a $0$ slope (the region in light blue) corresponds to the phase where there is a.s. exactly one connected component (a spanning tree). Outside this set there are a.s. infinitely many connected components (a spanning forest). This bounded set corresponds to a gaseous phase in the dimer model (as in Figure 1 and Figure 13) or reduces to a single point in the liquid phase.

\begin{figure}[H]
\centering
\includegraphics[width=0.5\textwidth]{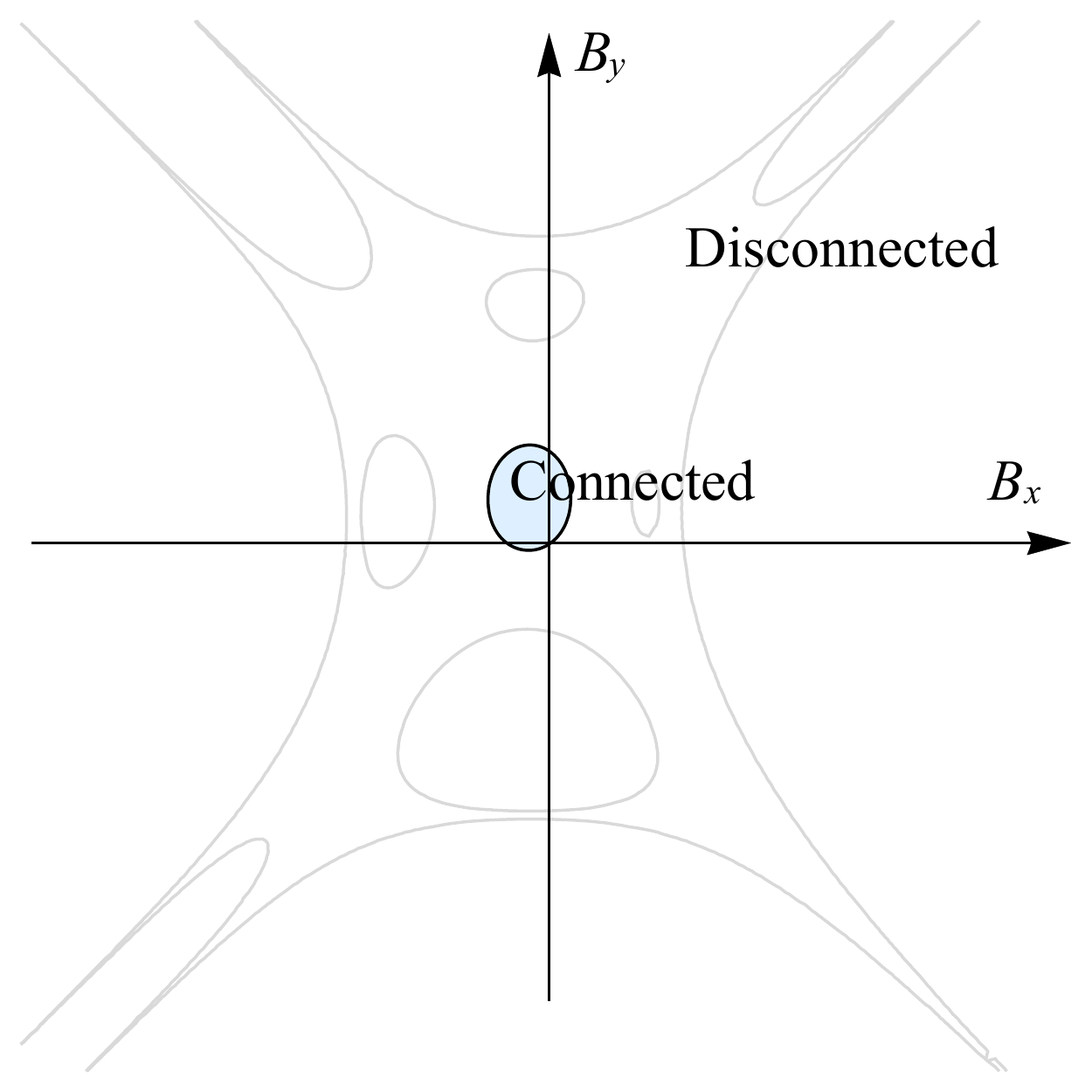}
\caption{Phase diagram of a square grid graph with fundamental domain $2\times2$.}
\end{figure}

There are also some interesting properties other than connectivity. Some are just repetitions of the results on the dimer model (\cite{KOS06}). In the liquid phase, the oriented edge-edge correlations decay polynomially and the variances of the height functions grow in the logarithm order. In the gaseous phases, the oriented edge-edge correlations decay exponentially, and the variances of the height functions are bounded. In the frozen phase, some of the height differences are deterministic.

\section{Remarks and open questions}
When we talk about the height, we mean the height function $h^M$ rather than $\tilde{h}^{(M,M_0)}$. The zero height change has a specific role in our problem.

The result that a measure of non-zero slope almost surely gives infinite connected components is true for any $\mathbb{Z}^2$-periodic graph. We conjecture that for general $\mathbb{Z}^2$-periodic graphs it is still true that there is only one connected component when slope is zero. Note that $0$ slope is an integer point in Newton polygon, if the weights are arbitrarily chosen, this is likely to correspond to a gaseous phase, and the origin $B=(0,0)$ lies on its boundary, see Lemma 5.1.

Measure corresponding to slope $0$ gives spanning trees whose branches are described by $LERW$. When slope is not $0$, there are bi-infinite bands. Inside such bands there are free spanning forests rooted at boundaries of bands, which are bi-infinite paths. It is interesting to see what such paths are.

When the slope is zero and the condition ($\star$) is verified, the toroidal dimer measure on $\mathcal{G}_N$ and the wired-spanning-tree measure on $G_N$ converge to the same limiting measure. By this fact we may conclude that their asymptotic entropies are the same. In fact, \cite{CKP01} states that the asymptotic entropy of a region depends loosely on the boundary height function. For spanning-tree measures on $G_N$, the boundary height function is given by the winding of a $LERW$ killed at boundary, and by geometric intuition this is about zero when renormalized by $N$. For the toroidal dimer measure in our case, the slope is zero.

\bibliographystyle{alpha}
\bibliography{mybib}
\end{document}